\documentclass[12pt]{amsart} 

\usepackage[colorlinks=true,linkcolor=blue]{hyperref} 
\usepackage{amsmath}          
\usepackage{mathtools}        
\usepackage{amsthm}           
\usepackage{amssymb}          
\usepackage[shortlabels]{enumitem} 
\usepackage{dcpic,pictexwd}   
\usepackage{tikz-cd}          
\usepackage{wrapfig}          
\usepackage{caption}          
\usepackage{csquotes}         
\usepackage{mathrsfs}         
\usepackage{verbatim}         
\usepackage{aurical}          
\usepackage{tensor}           
\usepackage[hang,flushmargin]{footmisc} 

\newtheorem{theorem}{Theorem}[section]  
\newtheorem{lemma}[theorem]{Lemma}    
\newtheorem{proposition}[theorem]{Proposition} 
\newtheorem{corollary}[theorem]{Corollary}

\theoremstyle{definition}   
\newtheorem{remark}[theorem]{Remark}
\newtheorem{definition}[theorem]{Definition}

\theoremstyle{definition}

\newtheorem*{theorem*}{Theorem}         
\newtheorem*{lemma*}{Lemma}              
\newtheorem*{proposition*}{Proposition}  
\newtheorem*{corollary*}{Corollary}
\newtheorem{question*}{Question}
\theoremstyle{definition}   
\newtheorem*{remark*}{Remark}
\newtheorem*{definition*}{Definition}
\newtheorem*{example*}{Example}
\newtheorem*{observation*}{Observation}
\newtheorem{claim*}{Claim}
\newtheorem*{fact*}{Fact}
\newtheorem*{exercise*}{Exercise}
\newtheorem*{note*}{Note}
\newtheorem*{notation*}{Notation}
\newtheorem*{conjecture*}{Conjecture}

\newcommand{\mb}[1]{\mathbb{#1}}
\newcommand{\mc}[1]{\mathcal{#1}}
\newcommand{\msc}[1]{\mathscr{#1}}

\newcommand{\set}[1]{\{#1\}}
\newcommand{\norm}[1]{\lVert #1 \rVert}

\newcommand{\id}{\mathrm{Id}}

\newcommand{\ddbar}{\partial \bar{\partial}}

\title[Almost holomorphic mappings]{On holomorphic mappings between almost Hermitian manifolds}
\author{Kirollos Masood}

\begin{document}
\maketitle

\begin{abstract}
    Our goal is to combine the techniques of Xiaokui Yang, Valentino Tosatti, and others to establish a Liouville-type result for almost complex manifolds. The transition to the non-integrable setting is delicate, so we will devote a section to discuss the key differences, and another to introduce the tools we will be using. Afterwards, we present a proof of our main theorem.
\end{abstract}

{\let\thefootnote\relax\footnote{\noindent{\emph{Mathematics Subject Classification (2010)}: 32Q60. }}}

\tableofcontents

\newpage

\section{Introduction}

An \emph{almost complex manifold} is a real manifold $M$, together with a bundle morphism $J:TM \to TM$ so that $J^2 = -\id$. An \emph{almost Hermitian manifold} is an almost complex manifold $(M,J)$ together with a Riemannian metric $g$ that is compatible with the almost complex structure: $g(J\cdot,J \cdot)=g$. As the name suggests, an almost complex structure generalizes the notion of a complex manifold. In Hermitian geometry, one typically uses the Chern connection, which coincides with the Levi-Civita only when $M$ is K{\"a}hler. In the almost Hermitian setting, we have the canonical connection, which generalizes the Chern connection, first introduced in \cite{ehresmann1951structures}. In this paper, we will be dealing with arbitrary almost Hermitian manifolds and all geometric quantities will be those corresponding to the canonical connection, unless otherwise specified.

The realm of almost Hermitian geometry is relatively unexplored. As such, one wonders what results in complex geometry can be extended or modified to apply to the almost complex setting. We will focus on proving the following result, an extension of Theorem 1.2 in \cite{2018arXiv181003276Y} to almost Hermitian manifolds.

\begin{theorem}
\label{main}
   Let $f:(M^{2m},J,h) \to (N^{2n},\tilde{J},g)$ be an almost holomorphic map between almost Hermitian manifolds equipped with their respective almost Chern connections, with $M$ compact. Assume that the holomorphic sectional curvatures of $M$ and $N$ are positive and non-positive (or non-negative and negative), respectively. Then $f$ is a constant map.
\end{theorem}

\section{The Almost Complex Structure}

To make sense of this theorem and its proof we provide, a proper foundation in almost complex geometry is needed. In particular, we should explain what tools we may carry over from the integrable setting, and how we will adapt.

\begin{remark}
\label{del}
    Recall that on a complex manifold, we have a splitting: $d = \partial + \bar{\partial}$, with $\bar{\partial}^2=0$. In the almost complex setting, the same does not hold, though we do have a similar structure.
    \begin{enumerate}
    
        \item
        We still have a decomposition $\mc{A} ^{r}(M)=\bigoplus_{p+q=r} \mc{A} ^{p,q}(M)$ and respective projections. With this in mind, we may still define $\partial:\mc{A}^{p,q}(M) \to \mc{A}^{p+1,q}(M)$ and $\bar{\partial}:\mc{A}^{p,q}(M) \to \mc{A}^{p,q+1}(M)$. More explicitly, simply apply the exterior derivative and project onto the appropriate subspace.
        
        \item
        Although we don't have $d = \partial + \bar{\partial}$ on the entire cohomology ring, this is clearly the case on the level of smooth functions. Moreover, 
        \begin{gather*}
            \partial f = \frac{df-i(df)J}{2}, \quad \mathrm{and} \quad
            \bar{\partial} f = \frac{df+i(df)J}{2}.
        \end{gather*}
    
        \item
        Given a smooth function $f:M \to \mb{R}$, we can still define $\ddbar f \in \mc{A}^{1,1}(M)$ and we have the explicit formula
        \begin{gather*}
            \ddbar f = \frac{d \bar{\partial}f+ d \bar{\partial}f(J \cdot, J \cdot)}{2}.
        \end{gather*}
    \end{enumerate}
\end{remark}

\begin{definition}
    Let $f:(M,J,h) \to (N,\tilde{J}.g)$ be a map between almost Hermitian manifolds. We say that $f$ is almost holomorphic if $(f_*)J=\tilde{J} (f_*)$.
\end{definition}

\noindent
Note then that this means $f$ takes $(1,0)$-vectors to $(1,0)$-vectors and of course the same holds for $(0,1)$-vectors. 

\begin{remark}
\label{hess}
    Given \emph{any} affine connection $\nabla$ on $TM$, one may define the Hessian of a smooth function $\nabla^2 f$ (also denoted by $\nabla df$), which is the second covariant derivative of $f$ with respect to the given connection. This Hessian will still be positive (negative) semi-definite at a local min (max) of $f$. However, the tensor will not be symmetric unless one uses a torsion-free connection.
\end{remark}

\begin{definition}
\label{hol1}
    Let $(M^{2n},J,g)$ be an almost Hermitian manifold.
    \begin{enumerate}
    
        \item 
        We say that a real vector field $W \in \Gamma(TM)$ is \emph{automorphic} if $\pounds_{W}J=0$.
        
        \item
        Suppose $J$ is integrable. A real vector field $W \in \Gamma(TM)$ is \emph{holomorphic} if at any $p \in M$, $W$ is holomorphic as a mapping from $\mb{C}^n \supseteq U  \to \mb{C}^n$ through the use of a holomorphic chart. Note that this definition is independent of the choice of such charts.
        
        \item
        Let $J$ be integrable and suppose $W \in \Gamma(T^{1,0}M)$ (in particular, $W$ is a complex vector field). Then we may locally write $W$ as $W^i \frac{\partial}{\partial z^i}$. We shall call $W$ \emph{complex-holmorphic} if for any $p \in M$ and any holomorphic chart $(U,\phi,\set{z^i})$ about $p$, each $W^i$ is a holomorphic function on $U$.
        
    \end{enumerate}
\end{definition}

\begin{proposition}
\label{holeq}
    Let $M$ be a complex manifold and $W$ a real vector field. Then the following are equivalent:
    \begin{enumerate}
    
        \item 
        $W$ is automorphic.
        
        \item
        $W$ is holomorphic.
        
        \item
        $\frac{W-iJW}{2}$ is complex-holomorphic.
        
    \end{enumerate}
\end{proposition}

\begin{proof}
    Since $J$ is integrable, we have a local \emph{coordinate} frame $\set{\frac{\partial}{\partial x^k}, \frac{\partial}{\partial y^k} \mid 1 \leq k \leq n}$ so that $J \frac{\partial}{\partial x^k} = \frac{\partial}{\partial y^k}$. We may then locally write $W = a^k \frac{\partial}{\partial x^k}+b^k \frac{\partial}{\partial y^k}$. A direct computation yields
    \begin{gather*}
        (\pounds_{W}J) \left(\frac{\partial}{\partial x^k} \right) =  \left[ W,J \frac{\partial}{\partial x^\ell} \right] -J \left[ W, \frac{\partial}{\partial x^\ell} \right]\\
        =\left[ a^k \frac{\partial}{\partial x^k}+b^k \frac{\partial}{\partial y^k}, \frac{\partial}{\partial y^\ell} \right]-J \left( -\frac{\partial a^k}{\partial x^\ell} \frac{\partial}{\partial x^k}-\frac{\partial b^k}{\partial x^\ell} \frac{\partial}{\partial y^k} \right) \\
        =-\frac{\partial a^k}{\partial y^\ell} \frac{\partial}{\partial x^k}-\frac{\partial b^k}{\partial y^\ell} \frac{\partial}{\partial y^k} -\left(-\frac{\partial a^k}{\partial x^\ell} \frac{\partial}{\partial y^k}+\frac{\partial b^k}{\partial x^\ell} \frac{\partial}{\partial x^k} \right) \\
        =\left(\frac{\partial a^k}{\partial x^\ell}-\frac{\partial b^k}{\partial y^\ell}   \right) \frac{\partial}{\partial y^k}- \left(\frac{\partial a^k}{\partial y^\ell} + \frac{\partial b^k}{\partial x^\ell} \right) \frac{\partial}{\partial x^k}.
    \end{gather*}
    This shows the equivalence between the first two conditions. We also note that applying the Lie derivative to each $\frac{\partial}{\partial y^\ell}$ instead gives us the same set of equations, so our first condition is not any stronger. Next, we establish the equivalence between the second and third conditions.
    \begin{gather*}
        \frac{W-iJW}{2} = \frac{1}{2} \left[ a^k \frac{\partial}{\partial x^k}+b^k \frac{\partial}{\partial y^k} -i a^k \frac{\partial}{\partial y^k} +i b^k \frac{\partial}{\partial x^k} \right] =\frac{\left[a^k+i b^k \right]}{2} \frac{\partial}{\partial z^k}
    \end{gather*}
    Our desired result follows from the following computation.
    \begin{gather*}
        \bar{\partial} \left(\frac{W-iJW}{2}\right)=
        \frac{1}{2} \frac{\partial (a^k+ib^k)}{\bar{z}^\ell}  d\bar{z}^\ell \otimes \frac{\partial}{\partial z^k} \\
        =\frac{1}{2} \left[ \left( \frac{\partial a^k}{\partial x^\ell} - \frac{\partial b^k}{\partial y^\ell} \right)+i\left( \frac{\partial a^k}{\partial y^\ell} + \frac{\partial b^k}{\partial x^\ell} \right) \right] d\bar{z}^\ell \otimes  \frac{\partial}{\partial z^k}
    \end{gather*}
\end{proof}

\noindent
With this in mind, we may use the term holomorphic to qualify real or complex vector fields, applying the appropriate definition in context.

\begin{remark}
\label{conn1}
    There are a few distinguished affine connections on $TM$ that can be used.
    \begin{enumerate}
    
        \item
        There is a unique torsion-free connection $\nabla^{LC}$, for which $\nabla^{LC} g=0$. This is called the \textit{Levi-Civita connection}.
        
        \item
        If $J$ is integrable, recall that $\bar{\partial}$ can be extended to act on vector fields. Then there is a unique connection $\nabla^{C}$, called the \textit{Chern connection}, for which $\nabla^{C} g = 0$ and $ \left[\nabla^{C}\right]^{0,1} = \bar{\partial}$.
        
        \item
        There is a unique connection $\nabla^{AC}$ for which $\nabla^{AC} g=0$, $\nabla^{AC} J=0$, and the corresponding torsion has vanishing (1,1) component. This connection is typically called the canonical connection, though we shall often use the term \textit{almost Chern connection}.
    \end{enumerate}
\end{remark}
\noindent
For our work on almost Hermitian manifolds, the last connection is more useful than the first. So henceforth, $\nabla$ will refer to the almost Chern connection, unless there is a superscript that indicates otherwise.

We may also consider what relations exist between these connections, and when they coincide. To understand the situation better, we need the following definition.
    
\begin{definition}
    An almost Hermitian manifold $(M,J,g)$ for which $\nabla^{LC}J=0$ is a \emph{K{\"a}hler manifold}.
\end{definition}

\noindent
When $M$ is K{\"a}hler, or even just complex, matters simplify considerably.

\begin{proposition}
\label{conn2}
    Let $\Omega = gJ \coloneqq g(J\cdot,\cdot)$.
    \begin{enumerate}
        
        \item
        If $\nabla^{LC} J=0$, then the Nijenhuis tensor $N_J$ vanishes, meaning $M$ is a complex manifold by the Newlander-Nirenberg Theorem.
        
        \item
        $\nabla^{LC} \Omega=0$ iff $\nabla^{LC} J=0$.
        
        \item
        $\nabla^{LC} J=0 \Rightarrow d\Omega=0$.
        
        \item
        If $N_J=0$, then $d\Omega=0 \iff \nabla^{LC} \Omega=0  $
        
        \item
        The Levi-Civita and Chern connections coincide iff $M$ is K{\"a}hler.
        
        \item
        If $M$ is complex, then the almost Chern and Chern connections coincide. 
    \end{enumerate}
\end{proposition}

\noindent
Many of the statements in Proposition \ref{conn2} follow quite readily, but we will not give the proofs here. The reader may consult \cite[$\S$5]{moroianu2007lectures}. We also have the following simplification in the K{\"a}hler case.

\begin{proposition}
\label{hol2}
    Suppose $M$ is K{\"a}hler and $W \in \Gamma(TM)$. Then $\pounds_W J=0$ iff $[J,\nabla W]=0$. 
\end{proposition}

\begin{proof}
    Pick any other vector field $X$. Since $J$ is parallel,
    \begin{gather*}
        (\pounds_W J)(X)=[W,JX]-J[W,X] \\
        =\nabla_W (JX) - \nabla_{JX} W  -J(\nabla_W X- \nabla_{X} W)\\
        =J \nabla_X W - \nabla_{JX} W = (J \circ \nabla W ) (X)- (\nabla W \circ J)(X) =  [J,\nabla W](X).
    \end{gather*}
    The result follows.
\end{proof}

\begin{proposition}
\label{del2}
    Using the almost Chern connection, $\ddbar f(e_j,\bar{e}_k)=\nabla^2 f(e_j, \bar{e}_k)$, where $e_j$ and $e_k$ are $(1,0)$-vectors.
\end{proposition}

\begin{proof}
    Select a local $(1,0)$-frame $\set{e_j}$ (with corresponding coframe $\set{\theta^j}$), Using the formula in Definition $\ref{del}$,
    \begin{gather*}
        \ddbar f (e_j, \bar{e}_k) = e_j (\bar{\partial}f(\bar{e}_k))-\bar{e}_k ( \bar{\partial}f(e_j))-\bar{\partial} f ([e_j, \bar{e}_k]) \\
        = e_j (\bar{e}_k (f))-0-\bar{\partial} f (\nabla_{e_j}\bar{e}_k - \nabla_{\bar{e}_k} e_j - \Theta(e_j,\bar{e}_k)) \\
        =e_j (\bar{e}_k (f))-\bar{\partial} f (\nabla_{e_j}\bar{e}_k)=e_j (\bar{e}_k (f))- (\nabla_{e_j}\bar{e}_k)(f)=\nabla^2 f(e_j, \bar{e}_k).
    \end{gather*}
\end{proof}

\noindent
In our main proof, we will utilize a particular complex manifold: the projective bundle. Let us now introduce it.

\begin{remark}
\label{proj}
    Let $(M^{2n},J,g)$ be an almost Hermitian manifold.
    \begin{enumerate}
    
        \item
        Then there is an induced almost complex structure, $J_{TM}$ on $TM$.
        
        \item
        We may consider $TM\backslash \set{0}$ and take a quotient under multiplication by scalars and by $J$. This manifold shall be referred to as the projective bundle and denoted by $\mb{P}(TM)$.
        
        \item
        $J_{TM}$ naturally descends to an almost complex structure $J_{P}$ on $\mb{P}(TM)$.
        
    \end{enumerate}
\end{remark}

It is important to keep in mind that it is possible to express most, if not all, of the definitions, formulas, and proofs presented using nothing but real coordinates. Indeed, we have already seen this in Proposition \ref{holeq}. For the most part, we will stick to the complexified tangent bundle, though at times, mostly for clarity, we will state corresponding implications or distinctions in real coordinates.

\section{Pseudo and Quasi Holomorphic Vector Fields}

Earlier, we used an arbitrary $(1,0)$-frame which, considering the isomorphism between $T^{1,0}M$ and $TM$, is completely arbitrary. We wish to be selective with our frame to simplify computations. And although we do not have local holomorphic vector fields, we can mimic this behavior at a single point.

\begin{definition}
\label{pseudo1}
    Let $M$ be an almost Hermitian manifold.
    \begin{enumerate}
    
        \item 
        A $(1,0)$-vector field $W$ is said to \emph{pseudo holomorphic at $p \in M$} if $[\bar{X},W]^{1,0}(p)=0$ for any $(1,0)$-vector field $X$.
        
        \item
        A vector field $W$ is \emph{quasi holomorphic at $p \in M$} if it is psuedo holomorphic at $p$ and $[X,[\bar{Y},W]]^{1,0}(p)=0$ for any $(1,0)$-vector fields $X$ and $Y$ that are pseudo holomorphic at $p$.
        
        \item
        A vector field $W$ is \emph{normal pseudo (quasi) holomorphic at $p \in M$} if it is pseudo (quasi) holomorphic at $p$ and has the additional property that $(\nabla W)(p)=0$.
        
    \end{enumerate}
\end{definition}

\noindent
Of course, such frames do not help us if we are unable find them. With this in mind, we have the following proposition. A complete proof may be found in \cite{2012arXiv1209.5078Y}.

\begin{proposition}
\label{pseudo2}
    Let $M$ be an almost Hermitian manifold
    \begin{enumerate}
        
        \item 
        At an point $p \in M$, there exists a local $(1,0)$-frame that is pseudo holomorphic at $p$.
        
        \item
        At an point $p \in M$, there exists a local $(1,0)$-frame that is quasi holomorphic at $p$.
        
        \item
        Moreover, we may choose to make either of the frames above normal at $p$ as well.
    \end{enumerate}
\end{proposition}

Recall that we already have the notion of a holomorphic vector field (real or complex) in Definition \ref{hol1}. And since we desire a property that focuses on a single point, we could have altered one of those definitions instead. As we shall see shortly, the vector fields in Definition \ref{pseudo1} will be more useful to us. However, a relation between the two exists:

\begin{proposition}
\label{pseudo3}
    Let $X \in \Gamma(TM)$ and suppose $N_J=0$. Then $\pounds_X J(p)=0$ iff $X-iJX$ is pseudo holomorphic at $p$.
\end{proposition}

\begin{proof}
    We now translate some of our conditions from complex frames to real frames. Let us write $W=X-iJX$ and $V=Y-iJY$, and suppress the point of evaluation.
    \begin{gather*}
        \pounds_{\bar{V}} W= [\bar{V},W]^{1,0}=[\nabla_{\bar{V}}W-\nabla_W \bar{V}]^{1,0}=[\nabla_{\bar{V}}W] \\
        =\nabla_{Y+iJY}(X-iJX)=\nabla_Y X+\nabla_{JY} (JX)+i[\nabla_{JY} X -\nabla_Y (JX)]
    \end{gather*}
    And so we see that for $X$ to be the real part of a pseudo holomorphic vector field, it must be that case that (at $p$) $\nabla_Y X = - \nabla_{JY} (JX)$ for all real vector fields $Y$. Note that $\nabla_{JY} X = \nabla_{Y} (JX)$ follows from the first condition by replacing $X$ with $JX$. 
    
    On the other hand, we may express the Nijenhuis tensor using the torsion of any connection that makes $J$ parallel (see \cite{boris}).
    \begin{gather*}
        0=N(X,Y)=[JX,JY]-J[JX,Y]-J[X,JY]-[X,Y] \\
        = T(X,Y)+JT(JX,Y)+JT(X,JY)-T(JX,JY)
    \end{gather*}
    However, there is a benefit of using the canonical connection among all such connections. \begin{gather*}
        T(A-iJA,B+iJB)= \\
        \left[ T(A,B)+T(JA,JB) \right] +i\left[ T(A,JB)-T(JA,B) \right]
    \end{gather*}
    We see that $T^{1,1}$ vanishing means that composing with $J$ in both entries yields a minus sign. We may use this to simplify our expression for $N$.
    \begin{gather*}
        \frac{1}{2}N(X,Y)=T(X,Y)+JT(X,JY)=0
    \end{gather*}
    Finally, we turn our attention to our defining equation for an automorphic vector field.
    \begin{gather*}
        (\pounds_X J)(Y)=[X,JY]-J[X,Y] \\
        =\nabla_X (JY)-\nabla_{JY} X -T(X,JY)-J\nabla_X Y +J \nabla_Y X+JT(X,Y) \\
        = J \nabla_Y X -\nabla_{JY} X -T(X,JY)+JT(X,Y) \\
        = \nabla_Y (JX)-\nabla_{JY} X + \frac{1}{2} J(N(X,Y))=\nabla_Y (JX)-\nabla_{JY} X
    \end{gather*}
    We see that $X$ is automorphic at $p$ iff $\nabla_Y (JX)-\nabla_{JY} X$, which is the same condition we found earlier.
    
\end{proof}

\begin{corollary}
\label{hol3}
    If $M$ is complex and $W$ is pseudo holomorphic at every point, then $W$ is holomorphic.
\end{corollary}

\begin{proof}
    This follows immediately from Propositions \ref{pseudo3} and \ref{holeq}.
\end{proof}

Now we give nice properties of (normal) pseudo and quasi holomorphic frames.

\begin{lemma}
\label{gauge}
    Let $\set{e_i}$ be a pseudo holomorphic frame at $p$.
    \begin{enumerate}[label=(\alph*)]
        
        \item
        $\Gamma_{\bar{j} k}^\ell(p) =0$.
        
        \item
        $\Gamma_{j \bar{k}}^{\bar{\ell}}(p) =0$.
        
        \item
        $[e_j, \bar{e}_k](p)=0$.
        
        \item
        $\Gamma_{ij}^k(p)=e_i[g_{j\bar{\ell}}](p)g^{k \bar{\ell}}(p)$.
        
        \item
        $\Gamma_{\bar{i}\bar{j}}^{\bar{k}}(p)=\bar{e}_i[g_{\ell \bar{j}}](p)g^{\ell \bar{k}}(p)$.
        
        \item
        Suppose $\set{e_i}$ is quasi holomorphic at $p$. Then $[\nabla_{e_i} \nabla_{\bar{e}_j} e_k](p)=0$
        
        \item
        Suppose $\set{e_i}$ is quasi holomorphic at $p$. 
        
        \noindent
        Then $R_{i\bar{j}k\bar{\ell}}(p)= -\bar{e}_{j}[e_i[g_{k \bar{\ell}}]] + g^{a \bar{b}}e_i[g_{k\bar{b}}]\bar{e}_j[g_{a \bar{\ell}}]$.
        
        \item
        Suppose $\set{e_i}$ is normal quasi holomorphic at $p$. 
        
        \noindent
        Then $R_{i\bar{j}k\bar{\ell}}(p)= -\bar{e}_{j}[e_i[g_{k \bar{\ell}}]]$.

    \end{enumerate}
\end{lemma}

\begin{proof}
    Most of these results can be found in \cite{2012arXiv1209.5078Y}, though we will prove them here for completeness. We will often omit the point of evaluation to avoid clutter.
    \begin{enumerate}[label=(\alph*)]
        
        \item
        $\Gamma_{\bar{j} k}^\ell = \theta^{\ell}[\nabla_{\bar{e}_j} e_k]=\theta^{\ell}[0]=0$.
        
        \item
        Observe that $g(e_\ell, \nabla_{e_j} \bar{e}_k)=\overline{g(\nabla_{\bar{e}_j} e_k ,\bar{e}_\ell)}=\overline{g(0,\bar{e}_{\ell})}=0$, since $e_k$ is pseudo holomorphic. As this is true for any $e_\ell$ (and any $\bar{e}_\ell$, trivially), it must be that $\nabla_{e_j} \bar{e}_k=0$, and hence each $\Gamma_{j \bar{k}}^{\bar{\ell}}$.
        
        \item
        $[e_j, \bar{e}_k]=\nabla_{e_j} \bar{e}_k - \nabla_{\bar{e}_k} e_j - \Theta(e_j,\bar{e}_k)=0$, as each term evaluated at $p$ is $0$.
        
        \item
        We see that $T^{1,0}M$ and $T^{0,1}M$ are parallel subbundles of $T^{\mb{C}}M$. This, together with (b), yields
        \begin{gather*}
            e_i [g(e_j, \bar{e}_\ell)] = g(\nabla_{e_i} e_j, \bar{e}_\ell)+g(e_j, \nabla_{e_i} \bar{e}_\ell)= \Gamma_{ij}^k g_{k \bar{\ell}}+0.
        \end{gather*}
        Now rearrange to get the desired formula.
        
        \item
        We perform a similar computation as above, invoking (a) instead.
        \begin{gather*}
            \bar{e}_i [g(e_\ell, \bar{e}_j)] = g(\nabla_{\bar{e}_i} e_\ell, \bar{e}_j)+g(e_\ell, \nabla_{\bar{e}_i} \bar{e}_j)=0+\Gamma_{\bar{i} \bar{j}}^{\bar{k}} g_{j \bar{k}}.
        \end{gather*}
        Once again, simply rearrange and we are done.
        
        \item
        Keeping in mind that $\Theta^{1,1}=0$,
        \begin{gather*}
            \nabla_{e_i} \nabla_{\bar{e}_j} e_k=\nabla_{e_i} ([\bar{e}_j, e_k]+\nabla_{e_k} \bar{e}_j) \\
            = [e_i,([\bar{e}_j, e_k]+\nabla_{e_k} \bar{e}_j)]+\nabla_{([\bar{e}_j, e_k]+\nabla_{e_k} \bar{e}_j)} e_i.
        \end{gather*}
        When we evaluate at $p$, we may invoke (a), (b), and (c).
        \begin{gather*}
            \nabla_{e_i} \nabla_{\bar{e}_j} e_k = [e_i,[\bar{e}_j, e_k]+\nabla_{e_k} \bar{e}_j]=[e_i,[\bar{e}_j, e_k]]+[e_i,\nabla_{e_k} \bar{e}_j] \\
            =[e_i,[\bar{e}_j, e_k]]+\nabla_{e_i} \nabla_{e_k} \bar{e}_j - \nabla_{\nabla_{e_k} \bar{e}_j}  e_i - \Theta(e_i,\nabla_{e_k} \bar{e}_j) \\
            =[e_i,[\bar{e}_j, e_k]]+\nabla_{e_i} \nabla_{e_k} \bar{e}_j - \nabla_{0}  e_i - \Theta(e_i,0) \\
            =[e_i,[\bar{e}_j, e_k]]+\nabla_{e_i} \nabla_{e_k} \bar{e}_j
        \end{gather*}
        But $\nabla_{e_i} \nabla_{e_k} \bar{e}_j$ is also a $(0,1)$ vector, so
        \begin{gather*}
            [e_i,[\bar{e}_j, e_k]]^{1,0} = \left[\nabla_{e_i} \nabla_{\bar{e}_j} e_k-\nabla_{e_i} \nabla_{e_k} \bar{e}_j \right]^{1,0}=\nabla_{e_i} \nabla_{\bar{e}_j} e_k=0.
        \end{gather*}
        
        \item
        We perform a direct computation, using (c), (f), and then (d) and (e).
        \begin{gather*}
            R_{i \bar{j} k \bar{\ell}} = g(\nabla_{e_i} \nabla_{\bar{e}_j} e_k - \nabla_{\bar{e}_j} \nabla_{e_i} e_{k} - \nabla_{[e_i,\bar{e}_j]} e_k, \bar{e}_{\ell}) \\
            = -g(\nabla_{\bar{e}_j} \nabla_{e_i} e_{k}, \bar{e}_{\ell}) = -\bar{e}_j [g( \nabla_{e_i} e_{k}, \bar{e}_{\ell})]+g( \nabla_{e_i} e_{k}, \nabla_{\bar{e}_j} \bar{e}_{\ell}) \\
            = -\bar{e}_j [e_i [g(e_{k}, \bar{e}_{\ell})]-g(e_{k}, \nabla_{e_i} \bar{e}_{\ell})]+g( \nabla_{e_i} e_{k}, \nabla_{\bar{e}_j} \bar{e}_{\ell}) \\
            = -\bar{e}_j [e_i [g_{k \bar{\ell}} ]]+0+g(  e_{k}, \nabla_{\bar{e}_j} \nabla_{e_i} \bar{e}_{\ell})+g( \nabla_{e_i} e_{k}, \nabla_{\bar{e}_j} \bar{e}_{\ell}) \\
            = -\bar{e}_j [e_i [g_{k \bar{\ell}} ]]+
            \overline{g(\nabla_{e_j} \nabla_{\bar{e}_i} e_{\ell},\bar{e}_{k})}
            +g( \nabla_{e_i} e_{k}, \nabla_{\bar{e}_j} \bar{e}_{\ell}) \\
            = -\bar{e}_j [e_i [g_{k \bar{\ell}} ]]+g( \nabla_{e_i} e_{k}, \nabla_{\bar{e}_j} \bar{e}_{\ell}) 
            =-\bar{e}_j [e_i [g_{k \bar{\ell}} ]]+g_{a\bar{b}}\Gamma_{ik}^{a} \Gamma_{j \bar{\ell}}^{\bar{b}} \\
            = -\bar{e}_j [e_i [g_{k \bar{\ell}} ]]+g_{a\bar{b}} (e_i [ g_{k \bar{c}}] g^{a \bar{c}})(\bar{e}_j [g_{ d \bar{\ell}}] g^{d \bar{b}}) \\
            -\bar{e}_j [e_i [g_{k \bar{\ell}} ]]+ e_i [ g_{k \bar{b}}]\bar{e}_j [g_{ d \bar{\ell}}] g^{d \bar{b}}=-\bar{e}_j [e_i [g_{k \bar{\ell}} ]]+ g^{a \bar{b}} e_i [ g_{k \bar{b}}]\bar{e}_j [g_{a \bar{\ell}}] 
        \end{gather*}
        
        \item
        This follows immediately from (g) together with the condition that each $(\nabla e_i)(p)=0$, or equivalently, that each $d g_{i \bar{j}}=0$.

    \end{enumerate}
\end{proof}

\section{Proof of Main Theorem}

We are now ready to start our main proof, though we will take a brief, but necessary foray into Cartan's structure equations near the end.

\begin{proof}[Proof of Theorem \ref{main}.]
    We will adopt the convention of using the Greek alphabet to index objects in $M$ and the Roman alphabet to index objects in $N$. Moreover, we will be carrying out computations using local $(1,0)$-frames $\set{e_\alpha}$ and $\set{\tilde{e}_i}$ in $M$ and $N$. Let us also use $\set{\theta^\alpha}$ and $\set{\tilde{\theta}^i}$ for their respective coframes. The key to proving our main result is the function
    \begin{gather*}
        \msc{Y}=\frac{g(f_* W, f_* W)}{\norm{W}^2},
    \end{gather*}
    where $W \in TM\backslash \set{0}$. Note that this function is constant under rescaling and multiplication by $J$ (since $f$ is almost holomorphic). So we may view $\msc{Y}$ as a non-negative function defined on $\mb{P}(TM)$. We may use our frames to express this function locally.
    \begin{gather*}
        \msc{Y}=\frac{g_{i\bar{j}}f^i_\alpha \bar{f}_\beta^j W^\alpha \bar{W}^\beta}{h_{\gamma \bar{\delta}}W^{\gamma}\bar{W}^{\delta}}
    \end{gather*}
    We will also use a variety of identifications to make our calculations easier.
    \begin{gather*}
        \msc{H} \coloneqq h_{\gamma \bar{\delta}}W^{\gamma}\bar{W}^{\delta} \\
        \msc{F} \coloneqq  g_{i\bar{j}}f^i_\alpha \bar{f}_\beta^j W^\alpha W^\beta \\
        F^i \coloneqq f^i_\alpha W^\alpha \\
        \msc{Y} = \frac{\msc{F}}{\msc{H}}
    \end{gather*}
    In this manner, we may rewrite a few of our expressions:
    \begin{gather*}
        \msc{Y} = \frac{\msc{F}}{\msc{H}}, \qquad \msc{F} = g_{i\bar{j}}F^i \bar{F}^j.
    \end{gather*}
    Note that $F^i$ is nothing more than the $i$th component of $df(W)$. Now we perform a similar computation as in \cite{2018arXiv181003276Y}.
    \begin{gather}
        \ddbar{\msc{Y}}= \ddbar \left( \frac{\msc{F}}{\msc{H}} \right) = \partial \left[ \frac{(\bar{\partial} \msc{F})}{\msc{H}} - \frac{\msc{F} (\bar{\partial} \msc{H})}{\msc{H}^2} \right] \\
        =\frac{(\ddbar \msc{F})}{\msc{H}}-\frac{(\partial \msc{H}) \wedge (\bar{\partial}  \msc{F})}{\msc{H}^2}-\frac{(\partial \msc{F}) \wedge (\bar{\partial}  \msc{H})}{\msc{H}^2}\\
        +2\frac{\msc{F} (\partial \msc{H}) \wedge (\bar{\partial}  \msc{H})}{\msc{H}^3}-\frac{\msc{F} (\ddbar \msc{H})}{\msc{H}^2} \\
        =\msc{Y} (\ddbar \log \msc{H}^{-1})+\frac{(\partial \msc{F}) \wedge (\bar{\partial}  \log \msc{H}^{-1})+(\partial \log \msc{H}^{-1}) \wedge (\bar{\partial}  \msc{F})}{\msc{H}}\\
        +\frac{\msc{F}(\partial \log\msc{H}) \wedge (\bar{\partial}  \log \msc{H})}{\msc{H}}+\frac{(\ddbar \msc{F})}{\msc{H}}
    \end{gather}
    
    The next step is a term-by-term analysis. The goal is to pick a suitable frame to simplify matters. As a matter of notation, if we write $H^{\alpha \bar{\beta}} \coloneqq W^\alpha \bar{W}^{\beta}$, then $\msc{H} = H^{\alpha \bar{\beta}} h_{\alpha \bar{\beta}}$. Observe that
    \begin{align}
        &\partial \log \msc{H} = \frac{\partial H^{\alpha \bar{\beta}} h_{\alpha \bar{\beta}} +  H^{\alpha \bar{\beta}} \partial h_{\alpha \bar{\beta}}}{\msc{H}}, \quad \mathrm{and} \\
        &\bar{\partial} \log \msc{H} = \frac{\bar{\partial} H^{\alpha \bar{\beta}} h_{\alpha \bar{\beta}} +  H^{\alpha \bar{\beta}} \bar{\partial} h_{\alpha \bar{\beta}}}{\msc{H}}.
    \end{align}
    Now let $q = (p,[W^1,\ldots,W^m]) \in \mb{P}(TM)$, where $[W^1,\ldots,W^m]$ is a homogeneous, complex coordinate with respect to $\set{e_\alpha}$, a normal quasi holomorphic frame at $p$. If we restrict ourselves to horizontal, complex tangent vectors at $q$, then the terms involving derivatives of $H^{\alpha \bar{\beta}}$ vanish. And since our frame is normal at $p$, $dh_{\alpha \bar{\beta}}=\partial h_{\alpha \bar{\beta}} + \bar{\partial} h_{\alpha \bar{\beta}}=0$. In other words, these forms vanish identically on horizontal tangent vectors.
    
    With these conditions, only the first and last terms survive in our formula for $\ddbar \msc{Y}$. Let us move on to the first term.
    \begin{gather}
        \ddbar log \msc{H}^{-1} = - \partial \left( \frac{\bar{\partial} \msc{H}}{\msc{H}} \right) = -\frac{\ddbar \msc{H}}{\msc{H}}+\frac{\partial \msc{H} \wedge \bar{\partial} \msc{H}}{\msc{H}^2} \\
        =-\frac{\partial H^{\alpha \bar{\beta}} \wedge \bar{\partial} h_{\alpha \bar{\beta}}+ \partial h_{\alpha \bar{\beta}} \wedge \bar{\partial} H^{\alpha \bar{\beta}}+H^{\alpha \bar{\beta}} \ddbar h_{\alpha \bar{\beta}}+h_{\alpha \bar{\beta}} \ddbar H^{\alpha \bar{\beta}}}{\msc{H}} \\
        +\frac{[H^{\alpha \bar{\beta}} \partial h_{\alpha \bar{\beta}}+h_{\alpha \bar{\beta}} \partial H^{\alpha \bar{\beta}}] \wedge [H^{\gamma \bar{\delta}} \bar{\partial} h_{\gamma \bar{\delta}}+h_{\gamma \bar{\delta}} \bar{\partial} H^{\gamma \bar{\delta}}] }{\msc{H}^2}
    \end{gather}
    We may impose the same restrictions as we did earlier, but instead of selecting an arbitrary horizontal direction, set $u = (W^1,\ldots,W^m, \underbrace{0,\ldots,0}_{m-1})$, where we have used the same coordinates present in the basepoint.
    \begin{gather}
        \ddbar log \msc{H}^{-1}(u, \bar{u})=-\frac{1}{\msc{H}}H^{\alpha \bar{\beta}} \ddbar h_{\alpha \bar{\beta}}(u,\bar{u})\\
         = -\frac{1}{\msc{H}} W^{\alpha} \bar{W}^{\beta} W^{\gamma} \bar{W}^{\delta}\ddbar h_{\alpha \bar{\beta}}(e_\gamma,\bar{e}_{\delta})
    \end{gather}
    We can take this further, using Proposition \ref{del2} and Lemma \ref{gauge}.
    \begin{gather}
        \ddbar h_{\alpha \bar{\beta}} (e_\gamma,\bar{e}_{\delta}) = \nabla^2(e_j, \bar{e}_{\bar{\delta}})=e_{\gamma} (\bar{e}_{\delta} (h_{\alpha \bar{\beta}})) - (\nabla_{e_{\gamma}} \bar{e}_{\delta})(h_{\alpha \bar{\beta}}) \\
        =e_{\gamma} (\bar{e}_{\delta} (h_{\alpha \bar{\beta}})) = \bar{e}_{\delta}( e_{\gamma} (h_{\alpha \bar{\beta}}))=-R_{\gamma \bar{\delta} \alpha \bar{\beta}}
    \end{gather}
    Putting this together (and renaming for aesthetic purposes),
    \begin{gather}
        \ddbar \log \msc{H}^{-1}(u, \bar{u})=\frac{1}{\msc{H}} W^{\alpha} \bar{W}^{\beta} W^{\gamma} \bar{W}^{\delta}R_{\alpha \bar{\beta} \gamma \bar{\delta}}.
    \end{gather}
    
    Finally, we turn our attention to (the numerator of) the last term of Line 5. As before, we will perform a general expansion first, and then simplify terms. For emphasis, terms that are not present in the complex setting are colored in red.
    \begin{align}
        \ddbar \msc{F} &=\ddbar g_{i\bar{j}}F^i \bar{F}^j & &+ (\partial F^i \wedge \bar{\partial} g_{i\bar{j}}) \bar{F}^j&  &+ {\color{red} (\partial \bar{F}^j \wedge \bar{\partial} g_{i\bar{j}})F^i} \\
        &+ (\partial g_{i\bar{j}} \wedge \bar{\partial} \bar{F}^j) F^i& &+ g_{i\bar{j}}(\partial F^i \wedge \bar{\partial} \bar{F}^j )& &+ {\color{red} g_{i\bar{j}} F^i (\ddbar \bar{F}^j)} \\
        &+ {\color{red} (\partial g_{i\bar{j}} \wedge \bar{\partial} F^i) \bar{F}^j}& &+ {\color{red} g_{i\bar{j}} (\ddbar F^i) \bar{F}^j}& &+ {\color{red} g_{i\bar{j}} (\partial \bar{F}^j \wedge \bar{\partial} F^i)}
    \end{align}
    As we did in the domain, we can choose a normal quasi holomorphic frame $\set{\tilde{e}_i}$ at $f(p)$. This will eliminate the second, third, fourth, and seventh terms on the RHS above.
    
    At this point, we must derive some additional formulas. Using (generic) frames, we have currently been writing $f_{*} e_{\alpha} = f^{i}_{\alpha} \tilde{e}_i$. But is also true that 
    \begin{gather}
        \eta \coloneqq f^{*}\tilde{\theta}^i = f^i_{\alpha} \theta^{\alpha}.
    \end{gather}
    As the exterior derivative commutes with pullbacks, this gives us two formulas for $d\eta$.
    \begin{gather}
        d \eta = df^{i}_{\alpha} \wedge \theta^\alpha + f^{i}_{\alpha} d\theta^\alpha= df^{i}_{\alpha} \wedge \theta^{\alpha} + f^{i}_{\alpha} [-\theta^{\alpha}_{\beta} \wedge \theta^{\beta} + \Theta^{\alpha}] \\
        d \eta= f^{*} d\tilde{\theta}^i = f^{*} [-\tilde{\theta}^{i}_{j} \wedge \tilde{\theta}^{j} + \tilde{\Theta}^{i}] =  -(f^{*} \tilde{\theta}^{i}_{j}) \wedge f^j_{\alpha} \theta^{\alpha} +  f^{*} \tilde{\Theta}^{i}
    \end{gather}
    Here, we have used the first structure equations for $M$ and $N$, and the fact that pullbacks distribute over wedge products. If we set these equations equal to each other and rearrange, we have
    \begin{gather}
        [df^{i}_{\alpha}+(f^{*} \tilde{\theta}^i_j )f^{j}_{\alpha} -f^i_{\gamma} \theta^{\gamma}_{\alpha}] \wedge \theta^{\alpha} = f^{*} \tilde{\Theta}^{i} - f^i_{\alpha} \Theta^{\alpha}.
    \end{gather}
    The RHS has no $(1,1)$ component, so we may write
    \begin{gather}
        df^{i}_{\alpha}+(f^{*} \tilde{\theta}^i_j )f^{j}_{\alpha} -f^i_{\gamma} \theta^{\gamma}_{\alpha} = f^{i}_{\alpha \beta} \theta^{\beta}
    \end{gather}
    for some smooth function $f^i_{\alpha \beta}$. Let us take a closer look at the connection 1-forms. One should verify that in any frame,
    \begin{gather}
        \theta^{\gamma}_{\alpha} = \left[\Gamma_{\beta \alpha}^{\gamma} \theta^{\beta}+ \Gamma_{\bar{\beta} \alpha}^{\gamma} \bar{\theta}^{\beta} \right]
    \end{gather}
    If one selects a pseudo (or quasi) holomorphic frame at $p \in M$ (as we did), then the connection 1-forms are entirely $(1,0)$ at $p$. The same argument holds for $\tilde{\theta}^i_j$ at $f(p)$. From Line 23, this means that $df^i_{\alpha}(p) = \partial f^{i}_{\alpha}(p)$. Or in other words, $\bar{\partial} f^i_{\alpha}(p)=0$. Combined with the restriction to horizontal vectors, this eliminates the ninth term in our formula for $\ddbar \msc{F}$. Only the first, fifth, sixth, and eighth terms remain, but not for long.
   
   The decomposition in Line 24 allows us to rewrite Line 23 in a convenient manner, where we will now stop explicitly writing the pullbacks.
   \begin{gather}
       df^{i}_{\alpha} = f^{i}_{\alpha \beta} \theta^{\beta}+f^i_{\gamma} \theta^{\gamma}_{\alpha}- \tilde{\theta}^i_j f^{j}_{\alpha} \\
       = f^{i}_{\alpha \beta} \theta^{\beta} + f^i_{\gamma} \left[\Gamma_{\beta \alpha}^{\gamma} \theta^{\beta} + \Gamma_{\bar{\beta} \alpha}^{\gamma} \bar{\theta}^{\beta} \right] - f^{j}_{\alpha} \left[\tilde{\Gamma}_{k j}^{i} \tilde{\theta}^{k}+ \tilde{\Gamma}_{\bar{k} j}^{i} \bar{\tilde{\theta}}^{k} \right] \\
       = \left[f^{i}_{\alpha \beta} \theta^{\beta}+f^i_{\gamma} \Gamma_{\beta \alpha}^{\gamma} \theta^{\beta} - f^{j}_{\alpha} \tilde{\Gamma}_{k j}^{i} \tilde{\theta}^{k} \right] +
       \left[ f^i_{\gamma} \Gamma_{\bar{\beta} \alpha}^{\gamma} \bar{\theta}^{\beta} - f^{j}_{\alpha} \tilde{\Gamma}_{\bar{k} j}^{i} \bar{\tilde{\theta}}^{k} \right] \\
       = \partial f^i_{\alpha}+\bar{\partial} f^{i}_{\alpha}
   \end{gather}
   This formula holds locally, not just at $p$.
   \begin{gather}
       d\bar{\partial} f^{i}_{\alpha} = df^{i}_{\gamma} \wedge \Gamma_{\bar{\beta} \alpha}^{\gamma} \bar{\theta}^{\beta} -df^{j}_{\alpha} \wedge \tilde{\Gamma}_{\bar{k} j}^{i} \bar{\tilde{\theta}}^{k}+f^{i}_{\gamma} \Gamma_{\bar{\beta} \alpha}^{\gamma}  d \bar{\theta}^{\beta} -f^{j}_{\alpha} \tilde{\Gamma}_{\bar{k} j}^{i}  d \bar{\tilde{\theta}}^{k}+\\
       f^{i}_{\gamma} d\Gamma_{\bar{\beta} \alpha}^{\gamma} \wedge  \bar{\theta}^{\beta} - f^{j}_{\alpha} d \tilde{\Gamma}_{\bar{k} j}^{i}  \wedge \bar{\tilde{\theta}}^{k}
   \end{gather}
   Since we additionally chose our frames to be normal, all Christoffel symbols vanish at $p$ and $f(p)$, which eliminates the first four terms on the RHS of Line 29. This means
   \begin{gather}
       \ddbar f^{i}_{\alpha}(p) = \left[ d \bar{\partial} f^{i}_{\alpha} \right]^{1,1} (p) = \left[f^{i}_{\gamma}  (\partial \Gamma_{\bar{\beta} \alpha}^{\gamma} \wedge  \bar{\theta}^{\beta} ) - f^{j}_{\alpha} (\partial  \tilde{\Gamma}_{\bar{k} j}^{i}  \wedge \bar{\tilde{\theta}}^{k} ) \right](p)
   \end{gather}
   We have only to understand what the derivatives of the mixed Christoffel symbols are at $p$ and $f(p)$.
   \begin{gather}
       e_{\mu} \left[ \Gamma^{\gamma}_{\bar{\beta} \alpha}  \right]= e_{\mu} \left[ \theta^{\gamma}(\nabla_{\bar{e}_{\beta}} e_{\alpha}) \right] = [\nabla_{e_\mu} \theta^{\gamma}](\nabla_{\bar{e}_{\beta}} e_{\alpha})+ \theta^{\gamma} (\nabla_{e_{\mu}} \nabla_{\bar{e}_{\beta}} e_{\alpha})
   \end{gather}
    But $\nabla_{\bar{e}_{\beta}} e_{\alpha}=0$ at $p$ because our frame is pseudo holomorphic and $\nabla_{e_{\mu}} \nabla_{\bar{e}_{\beta}} e_{\alpha}=0$ at $p$ precisely because our frame is also quasi holomorphic. In a similar fashion, 
    \begin{gather}
       e_{\mu} \left[ \tilde{\Gamma}^{i}_{\bar{k} j} \circ f  \right] = f^{\ell}_{\mu} e_{\ell} \left[ \tilde{\Gamma}^{i}_{\bar{k} j} \right]=0.
   \end{gather}
    Thus the eighth term perishes. The reader is left to verify that the sixth term also vanishes under our assumptions. We turn our attention to the first term in our formula for $\ddbar \msc{F}$.
    \begin{gather}
        \ddbar (g_{ij} \circ f)(e_{\gamma},\bar{e}_{\delta}) = e_{\gamma} (\bar{e}_{\delta} (g_{i \bar{j}} \circ f )) - (\nabla_{e_{\gamma}} \bar{e}_{\delta} )(g_{i \bar{j}} \circ f) \\
        = e_{\gamma} (\bar{e}_{\delta} (g_{i \bar{j}} \circ f )) =e_{\gamma} (\bar{f}^{\ell}_{\delta} \bar{e}_{\ell}(g_{i \bar{j}})) \\= e_{\gamma} (\bar{f}^{\ell}_{\delta})  \bar{e}_{\ell}(g_{i \bar{j}})+ \bar{f}^{\ell}_{\delta} f^{k}_{\gamma} e_{\gamma}(\bar{e}_{\ell}(g_{i \bar{j}})) \\
        =(\partial \bar{f}^{\ell}_{\delta})(e_{\gamma})  \bar{e}_{\ell}(g_{i \bar{j}})+ \bar{f}^{\ell}_{\delta} f^{k}_{\gamma} \bar{e}_{\ell}( e_{\gamma}(g_{i \bar{j}})) = 0-\bar{f}^{\ell}_{\delta} f^{k}_{\gamma} R_{k \bar{\ell} i \bar{j}} 
    \end{gather}
    
    With this analysis completed, we may give provide the main argument now. Suppose $\msc{Y}$ attains a non-zero maximum at some $q=(p,[W^1,\ldots,W^m]) \in \mb{P}(TM)$, where $[W^1,\ldots,W^m]$ is a complex, homogeneous coordinate with respect to a normal quasi holomorphic frame $e_\alpha$ at $p$. Then $(\ddbar \msc{Y})(q) \leq 0$. However, when we evaluate this form on $(u,\bar{u})$, where $u=(W^1,\ldots,W^m)$, also using a normal quasi holomorphic frame $\tilde{e}_i$ at $f(p)$, we get
    \begin{gather*}
        \ddbar \msc{Y} (u,\bar{u}) = \frac{\msc{Y}(q)}{\msc{H}} R_{\alpha \bar{\beta} \gamma \bar{\delta}} W^{\alpha} \bar{W}^{\beta} W^{\gamma} \bar{W}^{\delta} \\
        +\frac{F^{i} \bar{F}^j \ddbar g_{i \bar{j}}(u, \bar{u})+g_{i\bar{j}} (\partial F^i \wedge \bar{\partial} \bar{F}^j)(u, \bar{u})}{\msc{H}} \\
        = \frac{\msc{Y}(q)}{\msc{H}} R_{\alpha \bar{\beta} \gamma \bar{\delta}} W^{\alpha} \bar{W}^{\beta} W^{\gamma} \bar{W}^{\delta}-\frac{1}{\msc{H}} f^{i}_{\alpha} W^{\alpha} \bar{f}^j_{\beta} \bar{W}^{\beta} \bar{f}^{\ell}_{\delta} f^{k}_{\gamma} R_{k \bar{\ell} i \bar{j}} W^{\gamma} \bar{W}^{\delta} \\
        +\frac{1}{\msc{H}} g_{i\bar{j}} (\partial F^i \wedge \bar{\partial} \bar{F}^j)(u, \bar{u}) \\
        \geq \frac{\msc{Y}(q)}{\msc{H}} R_{\alpha \bar{\beta} \gamma \bar{\delta}} W^{\alpha} \bar{W}^{\beta} W^{\gamma} \bar{W}^{\delta}-\frac{1}{\msc{H}} R_{i \bar{j} k \bar{\ell}} f^{i}_{\alpha} \bar{f}^j_{\beta} f^{k}_{\gamma} \bar{f}^{\ell}_{\delta} W^{\alpha} \bar{W}^{\beta} W^{\gamma} \bar{W}^{\delta} > 0.
    \end{gather*}
    We have reached a contradiction and thus conclude that $\msc{Y} \equiv 0$. Consequently, $f$ is constant, as the components of $f_{*}$ must be identically 0 (in any frame) due to $g$ being positive definite.
\end{proof}

\section{Acknowledgements}

I would like to thank Dr. Fangyang Zheng for suggesting this project, and for his tips for when I would run into difficulty. I would also like to thank my thesis advisor Dr. Andrzej Derdzinski for his advice and constant reminders to first understand the situation from a Riemannian perspective.

\newpage

\bibliographystyle{alphaurl}
\nocite{*}
\bibliography{holomorphicRefs}

\end{document}